\documentclass[11pt, a4paper]{amsart}
\usepackage[margin=2.7cm]{geometry}
\usepackage{amsthm}
\usepackage[colorlinks=true,citecolor=cyan]{hyperref}
\usepackage{amsmath}
\usepackage{amsfonts}
\usepackage{amssymb}
\usepackage{stmaryrd} 
\usepackage{amssymb} 
\usepackage{mathrsfs} 
\usepackage{amsfonts} 
\usepackage{esint} 
\usepackage{array} 
\usepackage{latexsym}
\usepackage{lmodern} 
\usepackage{graphicx} 
\usepackage{graphics} 
\usepackage{epsfig} 
\usepackage{color} 
\usepackage[dvipsnames]{xcolor}
\usepackage[all]{xy}
\usepackage[new]{old-arrows}
\usepackage{textcomp}
\usepackage{stmaryrd}
\DeclareGraphicsExtensions{.eps,.pdf,.jpeg,.png} 
\usepackage{fancyhdr}
\usepackage{multirow} 
\usepackage{comment}
\usepackage{pgf,tikz}
\usepackage{paralist}
\usepackage{faktor}
\usepackage{hyperref}
\usepackage{dirtytalk}
\usepackage{url}
\usepackage{hyperref}
\usepackage{blkarray,bigstrut}
\usepackage{nicematrix}
\usepackage{enumerate}

\makeatletter
\renewcommand*\env@matrix[1][*\c@MaxMatrixCols c]{%
  \hskip -\arraycolsep
  \let\@ifnextchar\new@ifnextchar
  \array{#1}}
\makeatother

\newtheorem{proposition}{Proposition}[section]
\newtheorem{theorem}{Theorem}[section]

\theoremstyle{definition}
\newtheorem{definition}{Definition}[section]
\newtheorem{remark}{Remark}[section]
\newtheorem{example}{Example}[section]
\newtheorem{notation}{Notation}[section]

\title{Lattice associated to a Shi variety}

\author{Nathan Chapelier-Laget}

\begin{document}

\maketitle

\begin{abstract}
Let $W$ be an irreducible Weyl group and $W_a$ its affine Weyl group. In \cite{NC1} the author defined an affine variety $\widehat{X}_{W_a}$, called the Shi variety of $W_a$, whose integral points are in bijection with $W_a$. The set of irreducible components of $\widehat{X}_{W_a}$, denoted $H^0(\widehat{X}_{W_a})$, is of some interest and we show in this article that  $H^0(\widehat{X}_{W_a})$ has a structure of a semidistributive lattice.
\end{abstract}

{
  \hypersetup{linkcolor=blue}
  \tableofcontents
}

\section{Introduction}
Let $V$ be a Euclidean space with inner product $( -, -)$. Let $\Phi$ be an irreducible crystallographic root system in $V$ with simple system $\Delta=\{\alpha_1, \dots, \alpha_n\}$. We set $m=|\Phi^+|$.  From now on, when we say \say{root system} it will always mean irreducible crystallographic root system. 

Let $W$ be the \emph{Weyl group} associated to $\mathbb{Z}\Phi$, that is the maximal (for inclusion) reflection subgroup of $O(V)$ admitting $\mathbb{Z}\Phi$ as a $W$-equivariant lattice.  
We identify $\mathbb{Z}\Phi$ and the group of its associated translations and we denote by $\tau_x$ the translation corresponding to $x \in \mathbb{Z}\Phi$.

 Let $k \in \mathbb{Z}$ and $\alpha \in \Phi$. Define the affine reflection $s_{\alpha,k} \in \text{Aff}(V)$ by $s_{\alpha,k}(x)=x-(2\frac{( \alpha, x )}{( \alpha, \alpha )}-k)\alpha$. We consider the subgroup $W_a $ of Aff($V$) generated by all affine reflections $s_{\alpha,k}$ with $\alpha \in \Phi$ and $k \in \mathbb{Z}$, that is 
$$
W_a := \langle s_{\alpha,k}~|~\alpha \in \Phi, ~k \in \mathbb{Z \rangle}.
$$
 The group $W_a$ is called the \emph{affine Weyl group} associated to $\Phi$. 
 
 Let $\alpha \in \Phi$ such that  $\alpha = a_1\alpha_1 + \cdots + a_n\alpha_n$ with $a_i \in \mathbb{Z}$. The height of $\alpha$ (with respect to $\Delta$) is defined by the number $h(\alpha) = a_1 + \cdots+ a_n$. We denote by $-\alpha_0$ the \emph{highest short root} of $\Phi$.
 
The set $ S_a := \{s_{\alpha_1},\dots,s_{\alpha_n}\} \cup \{s_{-\alpha_0,1}\} $  is a set of Coxeter generators of $W_a$. For short we will write $S_a = \{ s_0, s_1, \dots s_n\}$ where $s_0 := s_{-\alpha_0,1}$ and $s_i = s_{\alpha_i}$ for $i =1,\dots, n$.

It is also well known that $W_a = \mathbb{Z}\Phi \rtimes W$. Therefore, any element $w \in W_a$ decomposes as $w=\tau_x\overline{w}$ where $x \in \mathbb{Z}\Phi$ and $\overline{w} \in W$. The element $\overline{w}$  is called the \emph{finite part} of $w$.

Let $\alpha \in \Phi$ and $\alpha^{\vee}:= \frac{2\alpha}{( \alpha, \alpha )}$. For any $k \in \mathbb{Z}$ and any $m \in \mathbb{R}$, we set the hyperplanes 
\begin{align*}
H_{\alpha,k} &= \{x \in V~|~s_{\alpha,k}(x)=x \} \\
& = \{ x \in V~|~ ( x, \alpha^{\vee} ) = k\},
\end{align*}

\noindent the strips
\begin{align*}
H_{\alpha,k}^m & = \{x \in V~|~k < ( x ,\alpha^{\vee} ) < k+m \}.
\end{align*}
 
The collection of hyperplanes $H_{\alpha,k}$ is denoted by $\mathcal{H}(\Phi)$ or just $\mathcal{H}$ if there is no possible confusion. The fundamental polytope  $P_{\mathcal{H}}$  is defined by
 $$
 P_{\mathcal{H}} := \bigcap\limits_{\alpha \in \Delta}H_{\alpha,0}^1.
$$

An alcove of $V$ is by definition  a connected component of
$$
 V ~\backslash \bigcup\limits_{\begin{subarray}{c}
 ~ ~\alpha \in \Phi^{+} \\ 
  k \in \mathbb{Z}
\end{subarray}}
H_{\alpha,k}.
$$

We denote by $A_e$ the alcove $A_e = \bigcap_{\alpha \in \Phi^+} H_{\alpha,0}^1$. It turns out that $W_a$ acts regularly on the set of alcoves. Therefore we have a bijective correspondence between the elements of $W_a$ and all the alcoves. This bijection is defined by $w \mapsto A_w$ where $A_w := wA_e$. We call $A_w$ the corresponding alcove associated to $w \in W_a$. Any alcove of $V$ can be written as an intersection of special strips, that is there exists a $\Phi^+$-tuple of integers $(k(w,\alpha))_{\alpha \in \Phi^+}$ such that 
$$
A_w = \bigcap\limits_{\alpha \in \Phi^+}H_{\alpha, k(w,\alpha)}^1.
$$

\begin{definition}\label{def special point}
A point $x \in V$ is called special if $\text{Stab}_{W_a}(x)$ is isomorphic to $W$. Intuitively this notion embodies the points in $V$ that have the same geometry in their neighbourhood as the point 0.
\end{definition}

Proposition 10.17 of \cite{AB} tells us that such points exist. Moreover,  there exists a useful characterisation of these points:

\begin{proposition}[\cite{AB}, Proposition 10.19]\label{prop special point}
A point $x \in V$ is special if and only if every hyperplane in $\mathcal{H}$ is parallel to a hyperplane passing through  x.
\end{proposition}

In \cite{JYS1} Jian-Yi Shi shows that the $ \Phi^+$-tuple of integers $(k(w,\alpha))_{\alpha \in \Phi^+}$ subject to certain conditions characterizes entirely $w$ (we recall the details of this characterization in Section \ref{Section admitted}, which we refer to as the Shi's characterization). Built on this characterization, the author defined in \cite{NC1} an affine variety $\widehat{X}_{W_a}$, called the Shi variety of $W_a$, whose integral points are in bijection with $W_a$. We denote by $H^0(\widehat{X}_{W_a})$ the set of irreducible components of $\widehat{X}_{W_a}$.

The set $H^0(\widehat{X}_{W_a})$ has many interests that we describe now. It turns out that it is involved in several fields, a priori unrelated to the Shi varieties. 

First of all, we showed in \cite{NC1} that $H^0(\widehat{X}_{W_a})$ was parameterized by a collection of vectors in $\mathbb{Z}^m$, that we called \emph{admitted vectors} (see Section \ref{Section admitted}). We also showed that these vectors were exactly the $\Phi^+$-tuples of integers $(k(w,\alpha))_{\alpha \in \Phi^+}$ when $A_w$ lies in $P_{\mathcal{H}}$. 

When one is interested in $W(\widetilde{A}_n)$, the irreducible components of $\widehat{X}_{W(\widetilde{A}_n)}$ give many interesting results. The action by conjugation of $W(A_n)$ on itself is defined for all $\sigma, \gamma \in W(A_n)$ by ${\sigma.\gamma :=\sigma \gamma \sigma^{-1}}$.  Understanding the orbits of this action, which are the conjugacy classes, yielded a lot of research work in recent decades.
 We related in \cite{NC2} the conjugacy class of $(1~2~\cdots~n+1)$ with the irreducible components of the Shi variety corresponding to $W(\widetilde{A}_n)$, in particular we showed the following theorem

\begin{theorem}[\cite{NC2}, Theorem 1.3]\label{Bijection entre pyras et permutations}
There is a natural bijection between $H^0(\widehat{X}_{W(\widetilde{A}_n)})$ and the circular permutations (i.e. $(n+1)$-cycles) of $W(A_n)$. 
\end{theorem}

\begin{example}\label{example bij n=3}
The admitted vectors for $n=3$ are represented by a triangle where the coordinates are positioned in Figure \ref{position pyra}.
\end{example}

\medskip

\begin{figure}[h!]
\includegraphics[scale=0.7]{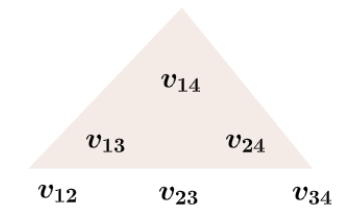} 
\caption{Coordinates of an admitted vector in $W(\widetilde{A}_3)$.}
\label{position pyra}
\end{figure}

Then, the bijection of Theorem \ref{Bijection entre pyras et permutations} can be seen below

\begin{figure}[h!]
\centering
\includegraphics[scale=0.8]{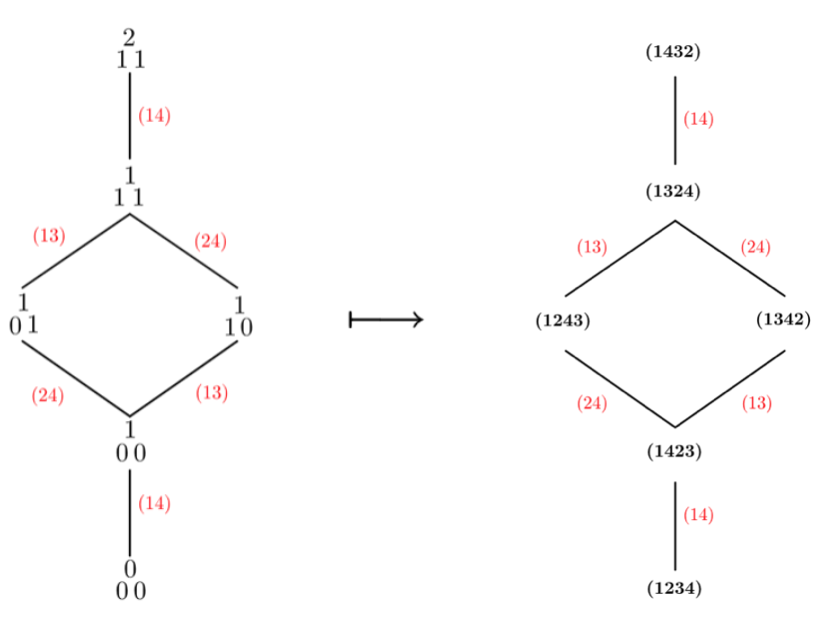} 
\caption{Poset of admitted vectors in $W(\widetilde{A}_3)$ on the left, and circular permutations of $W(A_3)$ on the right. In the expression of admitted vectors we drop the first line since the coefficients $v_{i,i+1}=0$ (see Definitions \ref{admissible vector} and \ref{admitted vector}). The red labels indicate, from left to right, the cover relation in the natural order on $\mathbb{Z}^6$; the conjugation action.} 
\label{image entre posets 1}
\end{figure}

In \cite{NC3} the authors also related $H^0(\widehat{X}_{W(\widetilde{A}_n)})$ to several other things, such as Eulerian numbers, $n$-gon, Young's lattice, and Reidemeister moves via the line diagrams. In particular we showed that $H^0(\widehat{X}_{W(\widetilde{A}_n)})$ has a structure of semidistributive lattice (see  \cite{NC3} Corollary 6.2) and we give a way to compute the join of any pair of two elements (see \cite{NC3} Section 4). 

It is then natural to ask whether the set $H^0(\widehat{X}_{W_a})$ has in general a structure of semidistributive lattice. The goal of this article is to give a positive answer to this question. Our main result is thereby the following theorem.

\begin{theorem}\label{lattice SD}
$H^0(\widehat{X}_{W_a})$ has a structure of semidistributive lattice.
\end{theorem}

     \begin{remark}
The poset of irreducible components of $\widehat{X}_{W_a}$ seems to be related to another poset introduced by D. Speyer et al in \cite{GQS}. In this article the authors study the looping case of Mozes’s game of numbers, which concerns the orbits in the reflection representation of affine Weyl groups situated on the boundary of the Tits cone. 
\end{remark}

\section{Generalities about Coxeter groups}\label{section cox}

\subsection{General definitions}\label{definitions gen}
Let $(W,S)$ be a Coxeter system with $e$ the identity element and $S$ the set of Coxeter generators. For $s,t \in S$ we denote $m_{st}$  the order of $st$.
Let $X$ be the $ \mathbb{R}$-vector space with basis
$\{e_s~|~s \in S\}$, and let $B$ be the symmetric bilinear form on $X$ defined by 
$$
B(e_s, e_t)=  \left\{
                          	\begin{array}{ll}
 						  -\text{cos}(\frac{\pi}{m_{st}})  & \text{if}~~ m_{st} < \infty \\
  				      	~~~~-1      & \text{if}~~m_{st} = \infty.
					    \end{array}
					    \right.
$$

  We denote by $O_B(X)$ the orthogonal group of $X$ associated to $B$. For each $s \in S$ we define $ \sigma_s : X \rightarrow X$ by $\sigma_s(x) = x - 2B(e_s,x)e_s$. The map $\sigma :W \hookrightarrow O_B(X)$ defined by $s \mapsto \sigma_s$ is called \emph{the geometrical representation} of $(W,S)$ (for more information the reader may refer to \cite{BOURB} ch. V, $\S$ 4 or \cite{Hum} ch 5.3). 
  Through this representation we identify $(W,S)$ with $(\sigma(W), \sigma(S))$.
  
  \begin{definition}\label{simple system}
   Let us denote 
 $ 
 \text{COS} :=  \{-1\} \cup \{-\text{cos}(\frac{\pi}{k}), k \in \mathbb{N}_{\geq 2}\}.
 $
 A simple system in $(X, B)$ is a finite subset $\Gamma$ in $X$ such that:
\begin{itemize}
\item[i)] $\Gamma$ is linearly independent;
\item[ii)] for all $\alpha, \beta \in \Gamma$ distinct, $B(\alpha, \beta) \in \text{COS}$;
\item[iii)] for all $\alpha \in \Gamma$, $B(\alpha, \alpha) = 1$.
\end{itemize}
\end{definition}

 We denote by $\Psi=W(\Gamma)$ the corresponding root system with basis $\Gamma$. Let us write $\Psi^+:= \Psi \cap \text{cone}(\Gamma)$ and $\Psi^- = -\Psi^+$. Then one has $\Psi=\Psi^- \sqcup \Psi^+$. If $\alpha \in \Psi$ we denote by $s_{\alpha}$ its corresponding  reflection. 
 
Let $\Gamma$ be a simple system in $(X, B)$. The group $W_{\Gamma} :=\langle s_{\alpha}~|~\alpha \in \Gamma \rangle$ is a subgroup of $W$. Moreover it is a Coxeter group with set of generators $S_{\Gamma} = \{s_{\alpha}~|~\alpha \in \Gamma\}$ (We refer the reader to \textcolor{blue}{\cite{Reflection_subgroups}} or \textcolor{blue}{\cite{SRLE}} Section 2.5 for more details about subreflection groups and their root system). We say that $\Gamma$ is a simple system for $(W_{\Gamma}, S_{\Gamma})$. In particular the set $\Delta :=\{e_s~|~s \in S\}$ is a simple system for $(W,S)$ and $S = S_{\Delta}$.

 The \emph{length function} $\ell : W \longrightarrow \mathbb{N}^*$ is defined as follows: $\ell(w)$ is the smallest number $r$ such that there exists an expression $w = s_{i_1}\dots s_{i_r}$ with ${s_{i_k} \in S}$. By convention, $\ell(e) = 0$. This function has been extensively studied and all basic  information about it can be found in \cite{BOURB} or \cite{Hum}. Let $w \in W$. An expression of $w$ is called a \emph{reduced expression} if it is a product of $\ell(w)$ generators. The \emph{inversion set} of $w$ is by definition
 \begin{align*}
 N(w) & := \{ \alpha \in \Psi^+~|~\ell(s_{\alpha}w) < \ell(w) \} \\
 &= \{ \alpha \in \Psi^+~|~w^{-1}(\alpha) \in \Psi^- \}.
 \end{align*}
 
Moreover we have $|N(w)| = \ell(w)$. In the case of affine Weyl groups, the length of an element $w \in W_a$ has a convenient interpretation in terms of its $\Phi^+$-tuple of integers $(k(w,\alpha))_{\alpha \in \Phi^+}$, namely 
$$
\ell(w) = \sum\limits_{\alpha \in \Phi^+}|k(w,\alpha)|.
$$

\subsection{Geometrical representation of $W_a$ and root system}\label{section tits} The goal of this section is to recall and give a good framework of the geometrical representation of affine Weyl groups.

 Let $\widehat{V} = V \oplus \mathbb{R}\delta$ with $\delta$ an indeterminate.  The inner product $(-,-)$ has a unique extension to a symmetric bilinear form on $\widehat{V}$ which is positive semidefinite and has a radical equal to the subspace $\mathbb{R}\delta$. This extension is also denoted $(-,-)$, and it turns out that the set of isotropic vectors associated to the form $(-,-)$ is exactly $\mathbb{R}\delta$. In particular for all $x, y \in V$ and for all $p,q \in \mathbb{Z}$ we have 
\begin{equation}\label{scalar}
(x + p\delta, y + q\delta) = (x, y).
\end{equation}

  The root system of $W_a$ is denoted $\Phi_a$ and its simple system is denoted $\Delta_a$. Using \cite{mDgL}  (Section 3.3 Definition 4 and Proposition 2) a concrete description of the affine (respectively, positive, simple) root system of $W_a$ is provided by:
\begin{align*}
\begin{split}
\Phi_a     &= \Phi^{\vee}+ \mathbb{Z}\delta, \\
\Phi_a^+ &= ((\Phi^{\vee})^+ +\mathbb{N}\delta) \sqcup ((\Phi^{\vee})^-+\mathbb{N}^*\delta), \\
\Delta_a &= \Delta^{\vee} \cup \{\alpha_0^{\vee}+ \delta\}.
\end{split}
\end{align*}

 \begin{remark}
 The link between $\widehat{V}$  and the geometrical representation is as follows. Let $\Delta_a = \{\alpha_i^{\vee}~|~i=1,\dots,n\}\cup \{\alpha_0^{\vee}+\delta\}$ be the simple system associated to $W_a$. To simplify the notations we denote $\lambda_i = \alpha_i^{\vee}$. We can now identify the $X$ of Section \ref{definitions gen} with $\widehat{V}$, by sending $e_{s_0}$ to $\frac{\lambda_0 + \delta}{||\lambda_0||}$ and $e_{s_i}$ to $\frac{\lambda_i}{||\lambda_i||}$ for $s_i \in S$. Since $\delta$ is isotropic for $(-,-)$ we only consider the scalar products $(\lambda_i, \lambda_j)$ for $i,j=0,\dots,n$. It is well known that $(\lambda_i, \lambda_j) = ||\lambda_i|| \cdot ||\lambda_j||\text{cos}(\theta)$ where $\theta$ is the angle between $\lambda_i$ and $\lambda_j$ in the plane generated by these two vectors. Moreover, it is also well known that $\theta = \pi-\frac{\pi}{m_{ij}}$. It follows that 
\begin{align}\label{inner prod and B}
(\lambda_i, \lambda_j) = ||\lambda_i||\cdot ||\lambda_j||\text{cos}(\pi-\frac{\pi}{m_{ij}})&=-||\lambda_i||\cdot||\lambda_j||\text{cos}(\frac{\pi}{m_{ij}})  \\&=||\lambda_i||\cdot||\lambda_j||B(e_{s_i}, e_{s_j}) \notag
\end{align}

 Furthermore we know that in the crystallographic root systems there are at most two root lengths. If $\lambda_i$ is short we have set before that $||\lambda_i||=1$. Therefore in the simply laced cases we have $(\lambda_i, \lambda_j) = B(e_{s_i}, e_{s_j})$. When $\lambda_i$ is longer than $\lambda_j$ we have two situations to look at: if $m_{ij} = 4$ then $||\lambda_i|| = \sqrt{2}||\lambda_j|| = \sqrt{2}$, and in particular $(\lambda_i, \lambda_j) = \sqrt{2}B(e_{s_i}, e_{s_j})$. If $m_{ij}= 6$ then $||\lambda_i|| = \sqrt{3}||\lambda_j|| = \sqrt{3}$ and it follows that $(\lambda_i, \lambda_j) = \sqrt{3}B(e_{s_i}, e_{s_j})$.

The geometrical representation sends the reflection $s_{\alpha,k}$ in $V$ to the reflection $s_{\alpha^{\vee}-k\delta}$ in $\widehat{V}$. In particular one can think of the hyperplane $H_{\alpha,k}$ as the fixed points of $s_{\alpha^{\vee}-k\delta}$. 
\end{remark}

%\begin{remark}
%Using (\ref{inner prod and B}) we can express the definition of simple system in $V$ in terms of the inner product $(-,-)$: $\Gamma \subset V$ is a simple system if and only if:
%\begin{itemize}
%\item[i)] $\Gamma$ is positively linearly independent;
%\item[ii)] for all $\alpha, \beta \in \Gamma$ distinct, $\frac{(\alpha, \beta)}{||\alpha||\cdot ||\beta||} \in \text{COS}$.
%\end{itemize}
%\end{remark}

\section{Background about the Shi variety}

\subsection{Admitted vectors}\label{Section admitted}
Let $\Phi$ be an irreducible crystallographic root system with simple system $\Delta = \{\alpha_1, \dots, \alpha_n \}$ and positive root system $\Phi^+ =\{\beta_1,\dots, \beta_m\}$. Let $W_a$ be the affine Weyl group corresponding to $\Phi$. 

 We recall in this section some necessary material. All the definitions were introduced in \cite{NC1}. We denote $\mathbb{Z}[X_{\Delta}] := \mathbb{Z}[X_{\alpha_1},\dots ,X_{\alpha_n}]$ and $\mathbb{Z}[X_{\Phi^+}] := \mathbb{Z}[X_{\beta_1},\dots ,X_{\beta_m}]$. For $w \in W_a$ and $Q \in \mathbb{Z}[X_{\Delta}]$ we denote 
$$
Q(w):=Q(k(w,\alpha_1),\dots ,k(w,\alpha_n)).
$$

\begin{notation} For two integers $a,b \in \mathbb{Z}$, the notation $\llbracket a,b \rrbracket$ means $[a,b] \cap \mathbb{Z}$.
If $Y \subset \mathbb{R}^m$ we denote by $Y(\mathbb{Z})$ the set of integral points of $Y$. 
\end{notation}
\bigskip

The following theorem is the Shi's characterization of the elements $w \in W_a$ by their $\Phi^+$-tuples of integers.

\bigskip

\begin{theorem}[\cite{JYS1}, Theorem 5.2]\label{thJYS1} 
Let $A = \bigcap\limits_{\alpha \in \Phi^+} H^1_{\alpha,k_{\alpha}}$ with $k_{\alpha} \in \mathbb{Z}$. Then $A$ is an alcove, if and only if, for all $\alpha$, $\beta \in \Phi^+$ satisfying  $\alpha + \beta \in \Phi^+$, we have the following inequality

\begin{equation}\label{inequation}
||\alpha||^2k_{\alpha} + ||\beta||^2k_{\beta} +1 \leq ||\alpha + \beta||^{2}(k_{\alpha+\beta} +1) \leq ||\alpha||^2k_{\alpha} + ||\beta||^2k_{\beta} + ||\alpha||^2+ ||\beta||^2 + ||\alpha+\beta||^2 -1.
\end{equation}
\end{theorem}

\bigskip

The following theorem decomposes the Shi coefficients as polynomial equations.
\begin{theorem}[\cite{NC1}, Theorem 4.1]\label{polynome}
Let $w \in W_a$. Then for all $\theta \in \Phi^+$ there exists a linear polynomial $P_{\theta} \in \mathbb{Z}[X_{\Delta}]$  with positive coefficients and  $\lambda_{\theta}(w) \in \llbracket 0, h(\theta^{\vee})-1\rrbracket$ such that 
\begin{equation} \label{eq:P}
k(w,\theta) = P_{\theta}(w) + \lambda_{\theta}(w).
\end{equation}
\end{theorem}

\medskip

\begin{definition}\label{affine var} Let $\theta \in \Phi^+$. Write $I_{\theta} := \llbracket 0, h(\theta^{\vee})-1 \rrbracket$. Notice that if $\theta$ is a simple root then $I_{\theta}=\{0\}$. For any root $\theta \in \Delta$ we set $P_{\theta} = X_{\theta}$ and $\lambda_{\theta}=0$. We denote by $P_{\theta}[\lambda_{\theta}]$ the polynomial $P_{\theta}+ \lambda_{\theta}-X_{\theta} \in \mathbb{Z}[X_{\Phi^+}]$. We define the ideal $J_{W_a}$ of $\mathbb{R}[X_{\Phi^+}]$ as ${J_{W_a} := \sum\limits_{\alpha \in \Phi^+} ~\langle \prod\limits_{\lambda_{\alpha} \in I_{\alpha}} P_{\alpha}[\lambda_{\alpha}] \rangle}$. We define  $X_{W_a}$ to be the affine variety associated to $J_{W_a}$, that is
$$
  X_{W_a} := V(J_{W_a}).
$$
\end{definition}

\begin{definition}\label{admissible vector}
 We say that $v=(v_{\alpha})_{\alpha \in \Phi^+} \in \mathbb{N}^m$ is an \emph{admissible vector} (or just \emph{admissible}) if it satisfies the boundary conditions, that is if for all $\alpha \in \Phi^+$ one has $v_{\alpha} \in I_{\alpha}$ (where $I_{\alpha}$ is defined in Definition \ref{affine var}). For instance, all the $\lambda:=(\lambda_{\alpha})_{\alpha \in \Phi^+}$ coming from the polynomials $P_{\alpha}[\lambda_{\alpha}]$ give rise to admissible vectors. Furthermore, each admissible vector arises this way. For short we will write $\lambda$ instead of $(\lambda_{\alpha})_{\alpha \in \Phi^+}$.
\end{definition}
 
 \begin{definition}
   Let $\lambda$ be an admissible vector. We denote by $J_{W_a}[\lambda]$ the ideal of $\mathbb{R}[X_{\Phi^+}]$ generated by the polynomials $P_{\alpha}[\lambda_{\alpha}]$ with $\alpha \in \Phi^+$, that is
  $
  J_{W_a}[\lambda]  := \langle P_{\alpha}[\lambda_{\alpha}],~\alpha \in \Phi^+\rangle.
 $		
 %Notice that we have the equality $J_{W_a}[\lambda] = \sum\limits_{\alpha \in \Phi^+} \langle P_{\alpha}[\lambda_{\alpha}] \rangle$.
 
 Then we define $X_{W_a}[\lambda]$ as the affine subvariety of $X_{W_a}$ associated to the ideal $J_{W_a}[\lambda]$, that is
$
X_{W_a}[\lambda] := V(J_{W_a}[\lambda]).
$
 \end{definition}
 
 \begin{definition}\label{admitted vector}
We denote $S[W_a]$ as the system of all the inequalities coming from Theorem \ref{thJYS1}, in other words a vector $v=(v_{\alpha})_{\alpha \in \Phi^+} \in \mathbb{R}^m$ is solution of $S[W_a]$ if whenever we have $\alpha$, $\beta \in \Phi^+$ satisfying $\alpha + \beta \in \Phi^+$ then (\ref{inequation}) is satisfied, i.e.
$$
||\alpha||^2v_{\alpha} + ||\beta||^2v_{\beta} +1 \leq ||\alpha + \beta||^{2}(v_{\alpha+\beta} +1) \leq ||\alpha||^2v_{\alpha} + ||\beta||^2v_{\beta} + ||\alpha||^2+ ||\beta||^2 + ||\alpha+\beta||^2 -1.
$$

 Let $\lambda$ be an admissible vector. We say that $\lambda$ is \emph{admitted} if it satisfies the system $S[W_a]$. 
\end{definition}

\begin{example}
We give in this example the set of inequations $S[W_a]$ for $W_a = W(\widetilde{A}_3)$. The positive root system of $A_3$ is given by the vectors  $e_i-e_j \in \mathbb{R}^4$ with $1 \leq i < j \leq 4$. We represent them in Figure \ref{rootA3} where it is easy to see when two positive roots have their sum that is also a positive root. 

\begin{figure}[h!]
\includegraphics[scale=4.5]{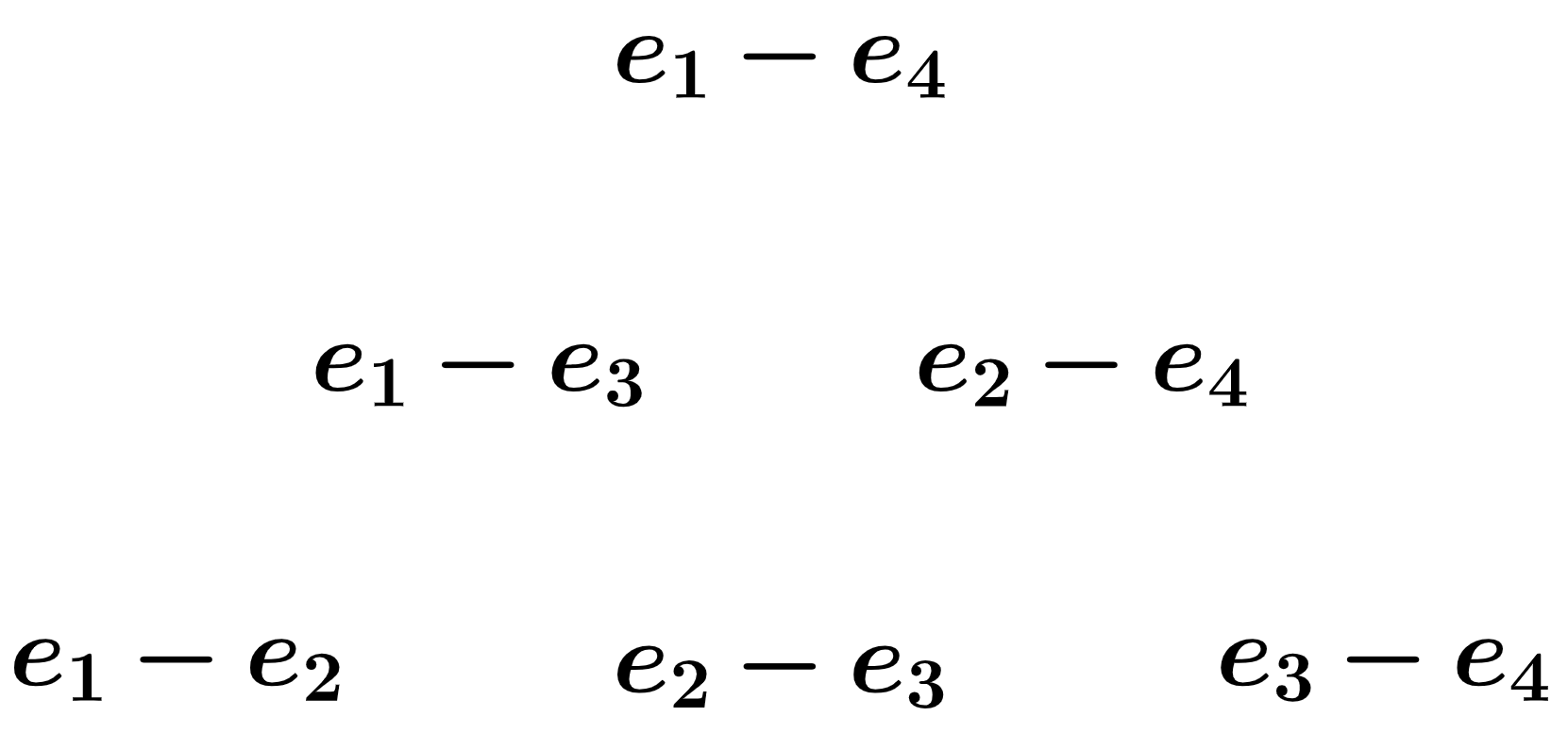} 
\caption{Positive roots of $A_3$ ordered by their hight.}
\label{rootA3}
\end{figure}

The roots of $A_3$ have all the same length which is $\sqrt{2}$. Therefore we can normalize the inequalities by $1/2$. It follows that a vector $v=(v_{ij}) \in \mathbb{Z}^6$  is solution of $S[W(\widetilde{A}_3)]$ if it satisfies the following inequalities
$$
\left\{
\begin{array}{ll}
 v_{12}+v_{23} \leq v_{13} \leq   v_{12}+v_{23} +1 \\
 v_{23}+v_{34} \leq v_{24} \leq   v_{23}+v_{34} +1 \\
 v_{12}+v_{24} \leq v_{14} \leq   v_{12}+v_{24} +1 \\
 v_{13}+v_{34} \leq v_{14} \leq   v_{13}+v_{34} +1 \\
\end{array}
\right.
$$
\end{example}

The next result gives the paramaterization of the elements of $H^0(\widehat{X}_{W_a})$ in terms of the admitted vectors.

\begin{theorem}[\cite{NC1}, Theorem 5.3]\label{TH central}
The map $\iota : {W_a \longrightarrow X_{W_a}(\mathbb{Z})}$ defined by $w \longmapsto (k(w,\alpha))_{\alpha \in \Phi^+}$ induces by corestriction a bijective map from $W_a$ to the integral points of a subvariety of $X_{W_a}$, denoted $\widehat{X}_{W_a}$, which we call the Shi variety of $W_a$. This subvariety is nothing but $ \widehat{X}_{W_a} = \bigsqcup\limits_{\lambda~\text{admitted}} X_{W_a}[\lambda]$. 

\begin{comment}
\begin{center}
\begin{tikzpicture} 

\node at (0, 0) {$W_a$} ;
\node at (4, 0) {$X_{W_a}(\mathbb{Z}	)$} ;
\node at (4, -2) {$\widehat{X}_{W_a}(\mathbb{Z}).$} ;
\node at (3, -0.8) {$\circlearrowleft$} ;
%\node at (4.15, -0.9) {$\iota$} ;
\node at (2, 0.2) {$\iota$} ; 
\node at (2,-0.99) {\rotatebox{-125}{$\wr$}} ;

\usetikzlibrary{arrows}
\draw [right hook->] (0.3,0) -- (3,0) ;
\draw [dashed, ->] (0.3,-0.2) -- (3,-1.7) ;
\draw [right hook ->] (4,-1.6) -- (4,-0.2) ; 

\end{tikzpicture}

\end{center}
\end{comment}
\end{theorem}

\subsection{Fundamental polytope $P_{\mathcal{H}}$}
In this section we recall some material about the polytope $P_{\mathcal{H}}$. These notions will be used in the proof of Theorem \ref{lattice SD}.

 Let $\mathbb{Z}\Phi^{\vee}$ be the coroot lattice and let us write $\mathbb{Z}\Phi^{\vee} = \mathbb{Z}\alpha_1^{\vee} \oplus \cdots\oplus \mathbb{Z}\alpha_n^{\vee}$. We define its dual lattice $(\mathbb{Z}\Phi^{\vee})^{*}$ as 
$$
(\mathbb{Z}\Phi^{\vee})^{*} := \{ x \in V ~|~(x, y ) \in \mathbb{Z} ~\forall y \in \mathbb{Z}\Phi^{\vee} \}.
$$
The lattice $(\mathbb{Z}\Phi^{\vee})^{*}$ is called the \emph{weight lattice}. This lattice has the following decomposition ${(\mathbb{Z}\Phi^{\vee})^{*}= \mathbb{Z}\omega_1\oplus \cdots\oplus \mathbb{Z}\omega_n}$ where $\omega_i$ is such that $( \alpha_i^{\vee}, \omega_j ) = \delta_{ij}$. The elements $\omega_i$ are called the \emph{fundamental weights} (with respect to $\Delta$).

 The fundamental weights $\omega_i$ are some of the vertices  of $P_{\mathcal{H}}$ and we have $P_{\mathcal{H}} = \{\sum\limits_{i=1}^nc_i\omega_i~|~ c_i \in \llbracket 0,1 \rrbracket \}$. Since $(\omega_i, \omega_j ) \geq 0$ for all $i,j$, the element of maximal norm in $P_{\mathcal{H}}$  is the vertex $\rho :=\sum\limits_{i=1}^n\omega_i$. Moreover, if $z \in \text{cone}(\Delta)$ we have $(z,\omega_i) \geq 0$ for all fundamental weight $\omega_i$. Finally, we define the set $$\text{Alc}(P_{\mathcal{H}}) := \{ w \in W_a~|~A_w \subset P_{\mathcal{H}} \}.$$

%It is known (see \textcolor{blue}{\cite{BOURB}} ch. VI, $\S$ 1, exercice 7) that the index of connection of $\Phi$ is the determinant of the Cartan matrix associated fo $\Phi$. Moreover, the Cartan matrix associated to $\Phi^{\vee}$ is the transpose of the Cartan matrix associated to $\Phi$. It follows that the indexes of connection of $\Phi$ and $\Phi^{\vee}$ are the same. 

 %Using (see  \textcolor{blue}{\cite{RichardKane}}, Section 11-6, Lemma C) we can deduce that $|\text{Alc}(P_{\mathcal{H}})|  = \frac{|W({\Delta})|}{f_{\Phi}}.$

Let $w \subset \text{Alc}(P_{\mathcal{H}})$. From the Shi's characterization it follows that $k(w,\alpha) = 0$ for all $\alpha \in \Delta$, and reciprocally, if $w' \in W_a$ is such that $k(w', \alpha) = 0$ for all $\alpha \in \Delta$ then $A_{w'} \subset P_{\mathcal{H}}$. The elements of this polytope seen as $\Phi^+$-tuple of integers are exactly the \emph{admitted} vectors and moreover a vector $\lambda \in \bigoplus\limits_{\alpha \in \Phi^+}\mathbb{R}\alpha$ is admitted if and only if there exists $w \in W_a$ such that $k(w,\alpha) = \lambda_{\alpha}$ for all $\alpha \in \Phi^+$ and such that $w \in \text{Alc}(P_{\mathcal{H}})$.  

\medskip

\begin{example} Let us take $W_a =W( \widetilde{B}_2)$ with simple system $\{\alpha_1, \alpha_2\}$. A short computation shows that $\omega_1 = \frac{1}{2}(2\alpha_1 + \alpha_2)$ and $\omega_2 = \alpha_1 + \alpha_2$.

\begin{figure}[h!]
\centering
\includegraphics[scale=0.5]{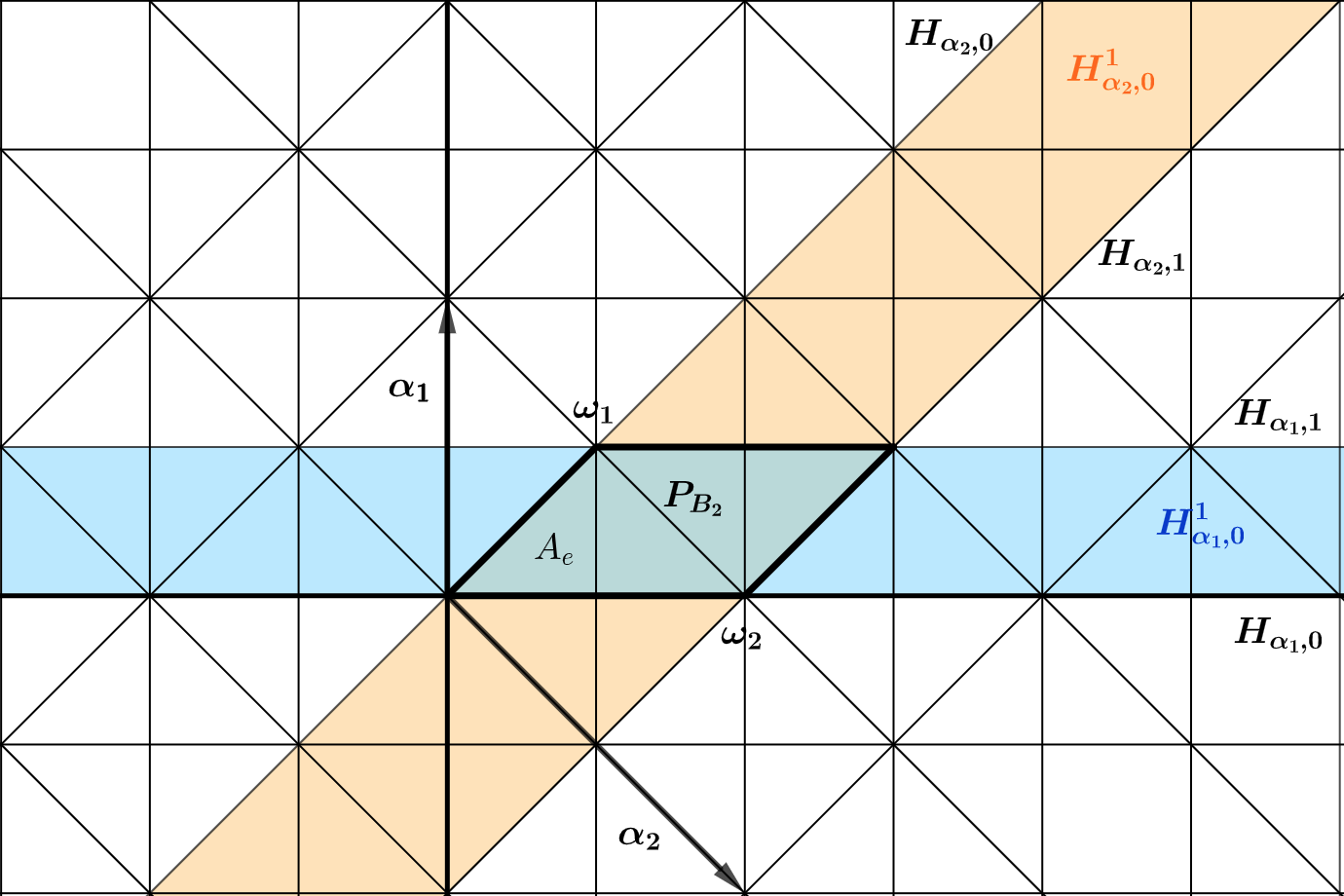} 
\caption{Fundamental parallelepiped $P_{B_2}$.}
\end{figure}

\end{example}

\section{Lattice structure on $H^0(\widehat{X}_{W_a})$}   \label{section poset}

\subsection{Poset structure on $H^0(\widehat{X}_{W_a})$}
In this section we define the natural poset structure on $H^0(\widehat{X}_{W_a})$ which we give in Definition \ref{ordre}. For $\lambda=(\lambda_{\alpha})_{\alpha \in \Phi^+}$ an admitted vector we denote by $w_{\lambda}$ the associated element of $\text{Alc}(P_{\mathcal{H}})$, that is $w_{\lambda}$ is such that $k(w_{\lambda}, \alpha) = \lambda_{\alpha}$ for all $\alpha \in \Phi^+$. Conversely, for $w \in \text{Alc}(P_{\mathcal{H}})$ we denote by $\lambda_w$ its corresponding Shi vector.

 \begin{definition}\label{ordre}
 The set of Shi vectors inherits the natural order of $\mathbb{Z}^m$. More precisely for two shi vectors $v=(v_{\alpha})_{\alpha \in \Phi^+}$ and $v'=(v'_{\alpha})_{\alpha \in \Phi^+}$ we say that $v \leq_S v'$ if and only if $v_{\alpha} \leq v'_{\alpha}$ for all $\alpha \in \Phi^+$. If $w, w'$ are the corresponding elements of $W_a$ associated to the Shi vectors $v, v'$ we write $w \leq_S w'$ as well.
 
 The set  $H^0(\widehat{X}_{W_a})$ has also a natural poset structure induced from the previous one. It is defined by ${X_{W_a}[\lambda] \leq_S  X_{W_a}[\gamma]}$ if and only if $\lambda \leq_S \gamma$.  %We will write either $\lambda \leq \gamma$ or ${X_{W_a}[\lambda] \leq  X_{W_a}[\gamma]}$.
  There is a minimal element in this poset which is the component corresponding to the admitted vector 0. %If $w$ and $w' \in P_{\mathcal{H}}$ we also say that $ w \leq w'$ if $k(w,\alpha) \leq k(w',\alpha)$ for all $\alpha \in \Phi^+$.
 The cover relation of $\leq_S$ is denoted by $\lessdot$ and it is given by $\lambda \lessdot \gamma$ if and only if there exists $\alpha \in \Phi^+$ such that $\gamma_{\alpha} = \lambda_{\alpha}+1$ and $\gamma_{\beta} = \lambda_{\beta}$ for $\beta \in \Phi^+\setminus \{\alpha\}$. 
 
 An interval for $\leq_S$ is denoted by $[-,-]_S$.
 \end{definition}  
 
 \medskip
 
 \begin{remark}\label{remark}
 The restriction of the right weak order $\leq_R$ of $W_a$ to $\text{Alc}(P_{\mathcal{H}})$  is exactly the same thing as the order $\leq_S$, that is for two admitted vectors $\lambda, \lambda'$ one has $ \lambda \leq_S \lambda' \Longleftrightarrow w_{\lambda} \leq_R w_{\lambda'}$. This claim follows easily from the fact that $w_{\lambda} \leq_R w_{\lambda'}$ if and only if their corresponding alcove share a face (of maximal dimension), which is equivalent to say that there exists $\alpha \in \Phi^+$ such that $\lambda'_{\alpha} = \lambda_{\alpha}+1$ and $\lambda'_{\beta} = \lambda_{\beta}$ for $\beta \in \Phi^+\setminus \{\alpha\}$.

 However, if we take two elements $w, w'$ such that their alcove don't live in the positive orthant then the equivalence $w \leq_S w' \Longleftrightarrow w \leq_R w'$ might fail.
\end{remark}

\medskip

\begin{example}
The polytope $P_{B_2}$ is as follows:
\begin{figure}[h!]
\centering
\includegraphics[scale=0.43]{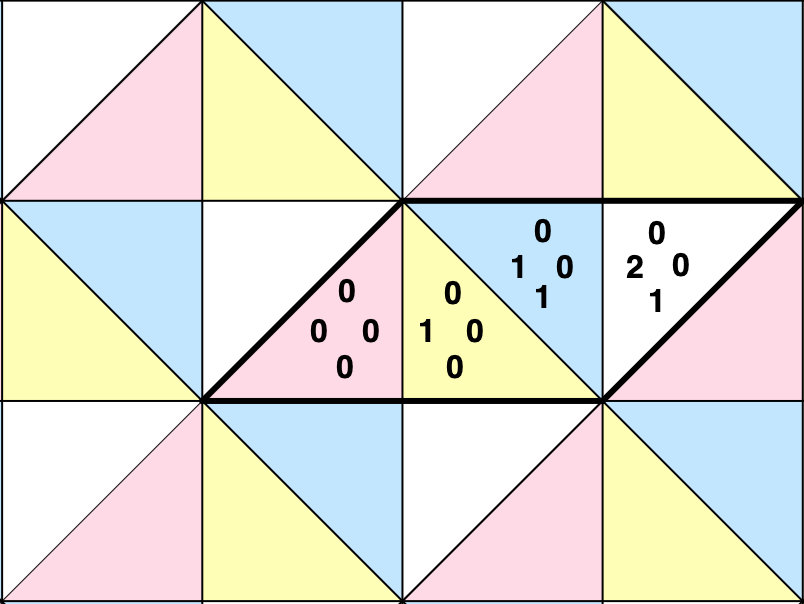} 
\caption{Polytope $P_{B_2}$ seen as set of representatives of irreducible components of $\widehat{X}_{W(\widetilde{B}_2)}$ (See Figure 9 of \cite{NC1} for more details about the colors).} 
\label{PB2}
\end{figure}

In Figure \ref{PB2} we denote the admitted vectors by dropping the two zeros corresponding to the simple roots, and by ordering the coordinates according to the height of the dual roots. Therefore, $H^0(\widehat{X}_{W(\widetilde{B}_2)})$ is as follows: 
\begin{figure}[h!]
\centering
\includegraphics[scale=0.63]{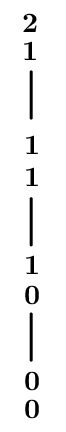} 
\caption{Poset associated to $\widehat{X}_{W(\widetilde{B}_2)}$.}
\label{PB2}
\end{figure}
\end{example}

\bigskip ~~
\bigskip

\begin{example}
Adapting Example \ref{example bij n=3} for $n=4$ we get the following presentation of an admitted vector

\begin{figure}[h!]
\centering
\includegraphics[scale=0.7]{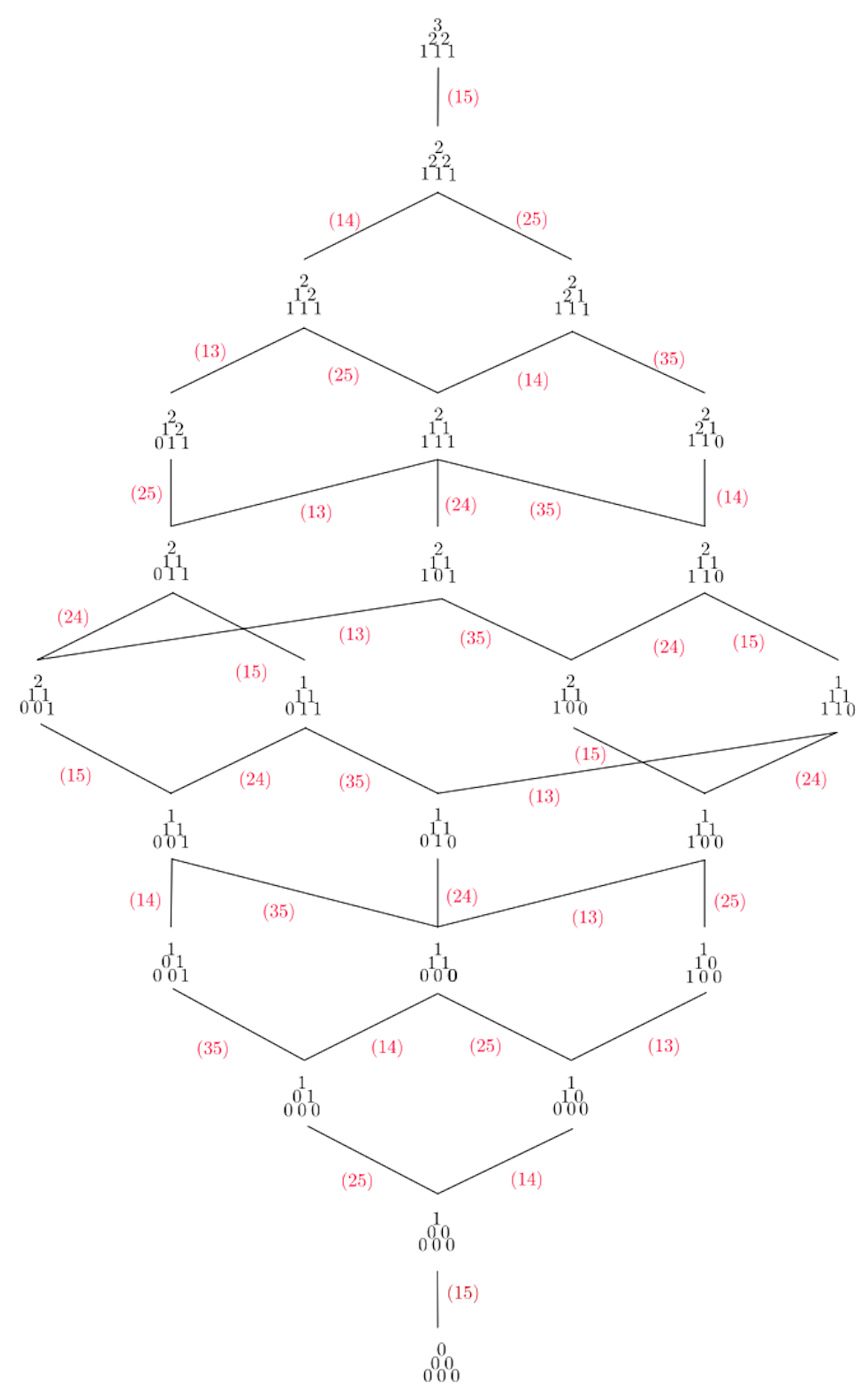} 
\caption{Poset associated to  $\widehat{X}_{W(\widetilde{A}_4)}$. The coordinates on the simple roots are erased since they are all equal to 0. The red labels represent the natural order on $\mathbb{Z}^{10}$.}
\label{pyra A4}
\end{figure}

\end{example}

 \newpage ~

\begin{example}
The positive roots of $B_3^{\vee}$ can be arranged according to their height into a shape looking like the temple of Kukulcan. Moreover the base is the set of dual simple roots. If $\lambda$ is an admitted vector, its coordinates on the dual simple roots are $0$.

\begin{figure}[h!]
\centering
\includegraphics[scale=0.55]{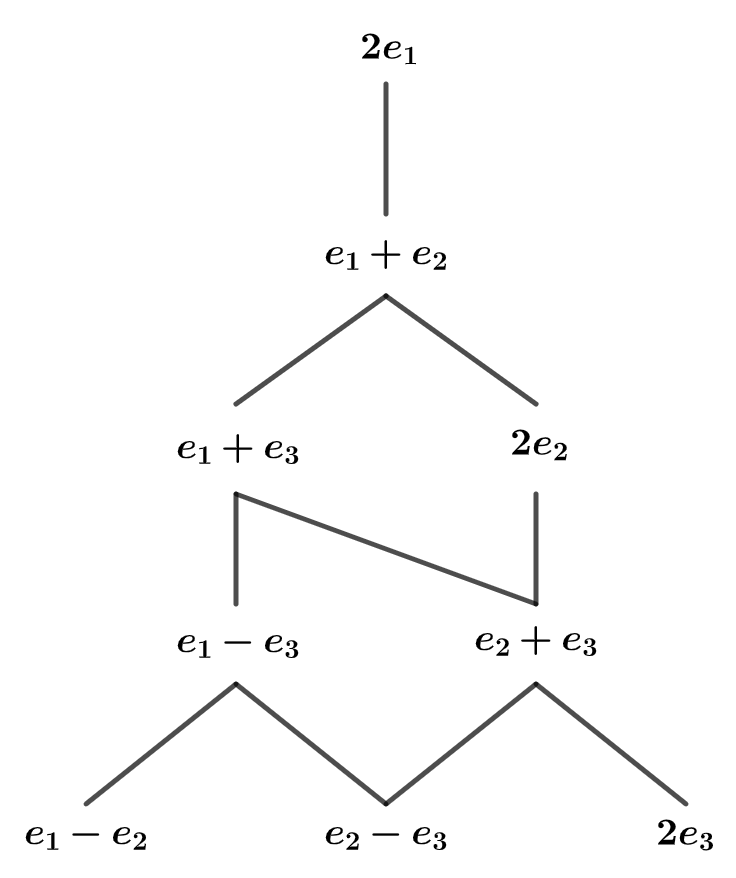} 
\caption{Positive roots of $B_3^{\vee}$.}
\label{immeuble}
\end{figure}

\bigskip ~~
\bigskip ~~
\bigskip ~~

\begin{figure}[h!]
\centering
\includegraphics[scale=0.8]{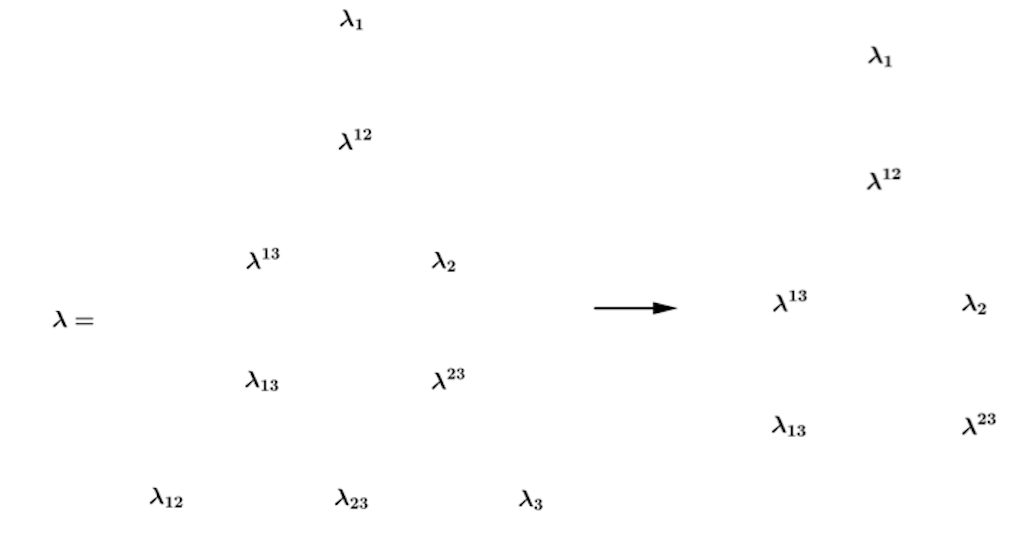} 
\caption{Presentation of an admitted vector $\lambda$ in $W(\widetilde{B}_3)$ where we erase the base.}
\label{immeuble}
\end{figure}

\end{example}

\begin{figure}[h!]
\centering
\includegraphics[scale=0.87]{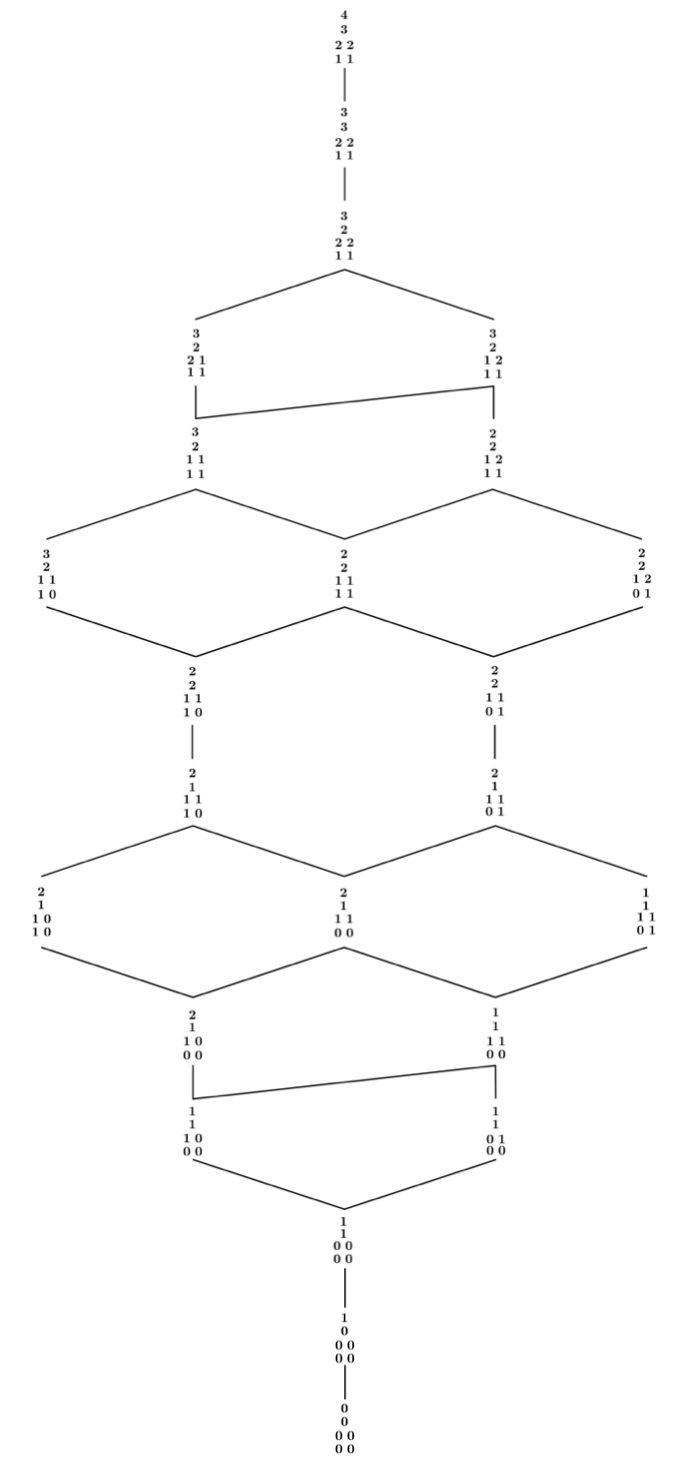} 
\caption{Poset associated to  $\widehat{X}_{W(\widetilde{B}_3)}$.}
\label{immeuble}
\end{figure}

\newpage

\subsection{Proof of the main result}

In this paragraph we recall some basics about lattices. A \emph{lattice} is a partially ordered set such that every pair $x, y$ of elements has a meet (greatest lower bound) $x \wedge y$ and join (least upper bound) $x \vee y$. A lattice is
distributive if the meet operation distributes over the join operation and the join
distributes over the meet.

A lattice $L$ is \emph{join semidistributive} if whenever $x, y, z \in L$ satisfy $x\vee y = x \vee z$, they also satisfy $x \vee (y \wedge z) = x \vee y$. This is equivalent to the following condition:
If $X$ is a nonempty finite subset of $L$ such that $x \vee y = z$ for all $x \in X$, then $(\bigwedge_{x \in X}x) \vee y = z$.  The lattice is \emph{meet semidistributive} if the dual condition
$(x\wedge y = x \wedge z) \Rightarrow (x \wedge(y \vee z) = x \wedge y)$ holds. Equivalently, if $X$ is a nonempty finite subset of $L$ such that $x \wedge y = z$ for all $x \in X$, then $(\bigvee_{x \in X}x) \wedge y = z$. The lattice
is \emph{semidistributive} if it is both join semidistributive and meet semidistributive.

\begin{proposition}\label{unique x}
There exists a unique alcove $A_w$ in $P_{\mathcal{H}}$ such that the point ${x:= \bigcap\limits_{\alpha \in \Delta}H_{\alpha,1}}$ is a vertex of $A_w$. Moreover, for $\alpha \in \Delta$ the hyperplanes $H_{\alpha,1}$ are some of the walls of $A_w$.
\end{proposition}

\begin{proof}
Let $W_x:=\langle s_{\alpha,1}, \alpha \in \Delta \rangle$, $\Delta_x:= \{ \alpha^{\vee}-\delta~|~\alpha \in \Delta \}$ and $\Phi_x := W_x(\Delta_x)$. It suffices to show two things: First the set  $\Delta_x$ is a simple system of $W_x$ and secondly $W_x = \text{Stab}_{W_a}(x)$.
Indeed, let us denote  $\mathcal{D}_x$ to be the simplicial cone pointed in $x$, cut out by the hyperplanes $H_{\alpha,1}$ for $\alpha \in \Delta$ and containing the alcove $A_e$. If $\Delta_x$ is a simple system of $W_x$ then $\mathcal{D}_x$ is the fundamental Weyl chamber of $W_x$, and if $W_x=\text{Stab}_{W_a}(x)$ then there is no hyperplane going through $x$ and $\mathcal{D}_x$. Thus, by setting $A_w$ to be the alcove with vertex $x$ and the $n-1$  walls $H_{\alpha,1}$ for $\alpha \in \Delta$ we have what we announced.

$\bullet$) Since $\Delta$ is linearly independent it follows that $\Delta_x$ is also linearly independent.  Because of Equation (\ref{scalar}) we know that $(\alpha^{\vee}-\delta, \beta^{\vee}-\delta)= (\alpha^{\vee}, \beta^{\vee})$ for all $\alpha, \beta \in \Delta$. Then, using Formula (\ref{inner prod and B}) we have $B(\alpha^{\vee}-\delta, \beta^{\vee}-\delta)= B(\alpha^{\vee}, \beta^{\vee})$  for all $\alpha, \beta \in \Delta$. 
Therefore, $\Delta_x$ is a simple system (in the sense of Definition \ref{simple system})  for $(W_x,S_x)$ where $S_x := \{s_{\alpha,1}~|~\alpha \in \Delta\}$.

$\bullet$) First of all it is clear that $W_x \simeq W$. Therefore it follows that $|\Phi_x^+| = |\Phi^+|$ and then the number of hyperplanes passing through $x$ is the same as the number of hyperplanes passing through 0. Moreover, we know  that each hyperplane of $\mathcal{H}$  is parallel to a hyperplane passing through 0, that is parallel to a hyperplane $H_{\alpha,0}$ with $\alpha \in \Phi^+$. In particular each hyperplane passing through $x$ is parallel to such a hyperplane. Therefore, it follows that each hyperplane of $\mathcal{H}$ is parallel to a hyperplane passing through $x$. Thus, Proposition \ref{prop special point} implies that $x$ is a special point, that is $\text{Stab}_{W_x}(x) \simeq W$. It follows then that $W_x \simeq \text{Stab}_{W_a}(x)$. Finally, since $W$ is finite, $W_x$ is also finite and then, since $W_x \subset \text{Stab}_{W_a}(x)$,  it follows that $W_x =\text{Stab}_{W_a}(x)$.
\end{proof}

%\begin{remark}
%Using the general theory of Coxeter groups it is well known that $\text{Stab}_{W_a}(x)$ is a parabolic subgroup of $W_a$, that is $\text{Stab}_{W_a}(x) =gW_Ig^{-1}$ for some $g \in W_a$ and some $I \subset S_a$. Actually it is not difficult to see that $I = \{s \in S_a~|~(g(\alpha_s), x) = 0 \}$. Finally, it is possible to show that the $w$ in Proposition \ref{unique x} is the minimal coset representative of $gW_I$.
%\end{remark}

\begin{proposition}\label{max}
Let $w$ as defined in Proposition \ref{unique x}. The set $H^0(\widehat{X}_{W_a})$ has a unique maximal element which is $X_{W_a}[\lambda_w]$.
\end{proposition}

\begin{proof}
First of all, we identify $H^0(\widehat{X}_{W_a})$ with the elements of $\text{Alc}(P_{\mathcal{H}})$ since these two sets are in bijection thanks to Theorem \ref{TH central}. Let $x$ be as in Proposition \ref{unique x}. The proof is split in 4 parts.

\vspace{0.25cm}

\textbf{-} We begin by showing that if $\text{Alc}(P_{\mathcal{H}})$ has a maximal element $g$ (in the sense of Definition \ref{ordre}) then the alcove of this element must have $x$ as vertex. Assume that it is not that case, namely $A_{g} \cap \{x\} = \emptyset$. Then, since Proposition \ref{unique x} implies that there exists a unique element $w \in \text{Alc}(P_{\mathcal{H}})$ having $x$ as vertex, it follows that a wall of $A_w$ is supported by an hyperplane $H_{\theta,k}$ (with $\theta \in \Phi^+ \setminus \Delta$ and $k>0$) that separates $A_w$ and $A_g$ and such that $A_g$ and $A_e$ are on the same side of $H_{\theta,k}$. Therefore, $k(g, \theta) < k(w,\theta)$ which is impossible by definition of $\leq_S$.

\vspace{0.25cm}

\textbf{-} \textbf{Main claim}.  We claim that $w$ is strictly greater than any other element in $\text{Alc}(P_{\mathcal{H}})$. We proceed by contradiction by assuming that $w$ is not the greater element of $\text{Alc}(P_{\mathcal{H}})$. It follows that we have a hyperplane $H_{\alpha,k}$ with $\alpha \in \Phi^+\setminus \Delta,~k \in \mathbb{N}$ that cuts $P_{\mathcal{H}}$ into two connected components such that $A_w$ and $A_e$ are in the same one and such that $x \notin H_{\alpha,k}$.  
 Let $A_{w'}$ be an alcove in the connected component that does not contain $A_e$. It follows that 
 \begin{equation}\label{contradiction}
 k(w,\alpha) < k(w',\alpha)
 \end{equation}
 
\textbf{-} \textbf{Intermediary claim with its proof}. Let $y$ be a point of $A_{w}$ and $y'$ be a point of $A_{w'}$. Therefore, since $y$ and $y' \in P_{\mathcal{H}}$ there exist $a_1,\dots ,a_n$ and $b_1,\dots ,b_n \in \mathbb{R}^+$ such that $y=a_1\alpha_1 +\cdots+a_n\alpha_n$ and $y'=b_1\alpha_1 +\cdots+ b_n\alpha_n$. 
We claim now that without loss of generality one can assume that $b_i \leq a_i$ for all $i \in \llbracket 1,n\rrbracket$. Let us explain this claim. Write $y$ and $y'$ in the basis of fundamental weights: 
${y=c_1\omega_1 + \cdots +c_n\omega_n}$ and $y' = d_1\omega_1 + \cdots + d_n\omega_n$ with $c_i$ and $d_i \in \mathbb{R}^+$. Since $y \in A_w$ and $y' \notin A_w$, and since $x$ is a vertex of $A_w$, we can take $y$ as close as we want to $x$. It follows here that there is no problem of assuming that $d_i \leq c_i$ for all $i$. Therefore, we make this assumption. It turns out that the inverse of the Cartan matrix $C^{-1}=(h_{ij})_{i,j \in \llbracket 1,n\rrbracket}$ of $W$ is the change-of-basis matrix of the basis of simple roots to the basis of fundamental weights. Moreover, it is known (see \cite{LT} or \cite{YY}) that all the coefficients of  $C^{-1}$ are positive. It follows that 
$$C^{-1} \begin{pmatrix}
c_1\\
\vdots \\
c_n
\end{pmatrix}
 = \begin{pmatrix}
a_1\\
\vdots \\
a_n
\end{pmatrix}~~ \text{and}~~ C^{-1}\begin{pmatrix}
d_1\\
\vdots \\
d_n
\end{pmatrix} = \begin{pmatrix}
b_1\\
\vdots \\
b_n
\end{pmatrix}.$$

Thus, the $i$-th coordinate of $y$ is $\sum\limits_{k=1}^nh_{ik}c_k$ and the $i$-th coordinate of $y'$ is $\sum\limits_{k=1}^nh_{ik}d_k$. Since ${d_i \leq c_i}$ for all $i$ and since $h_{ik} \geq 0$ for all $k=1,\dots ,n$ it follows that $\sum\limits_{k=1}^nh_{ik}d_k \leq \sum\limits_{k=1}^nh_{ik}c_k$. However the $i$-th coordinate of $y$ is nothing but $a_i$, and $i$-th coordinate of $y'$ is
$b_i$. Finally we have shown that $b_i \leq a_i$ for all $i$, which proves the intermediary claim.

\vspace{0.25cm}

\textbf{-} \textbf{Conclusion of the main claim}. We are now able to conclude the first claim. Let $\lfloor~ \rfloor$ be the floor function. From the way that the coefficients $k(-,-)$ are defined, one has $k(w,\alpha) = \lfloor(\alpha^{\vee}, y) \rfloor$ and $k(w',\alpha) = \lfloor(\alpha^{\vee}, y') \rfloor$. Via the expressions of $y$ and $y'$ one has:
\begin{align*}
(\alpha^{\vee}, y) &= (\alpha^{\vee}, a_1\alpha_1 + \cdots + a_n\alpha_n) = a_1(\alpha^{\vee}, \alpha_1) +\cdots+ a_n(\alpha^{\vee}, \alpha_n),  \\
(\alpha^{\vee}, y') &= (\alpha^{\vee}, b_1\alpha_1 + \cdots + b_n\alpha_n) = b_1(\alpha^{\vee}, \alpha_1) +\cdots+ b_n(\alpha^{\vee}, \alpha_n).
\end{align*}

It follows that  $0 \leq (\alpha^{\vee}, y') \leq (\alpha^{\vee}, y)$ and then $\lfloor(\alpha^{\vee}, y')\rfloor \leq  \lfloor(\alpha^{\vee}, y)\rfloor$, which means that $k(w',\alpha) \leq k(w,\alpha)$. This contradicts (\ref{contradiction}). Hence, $w$ must be the maximal element of $\text{Alc}(P_{\mathcal{H}})$. 
\end{proof}

\begin{proposition}
 Let $w$ be the element of $\text{Alc}(P_{\mathcal{H}})$ as defined in Proposition \ref{unique x}. The map $X_{W_a}[\lambda] \mapsto w_{\lambda}$ defines a poset isomorphism between $H^0(\widehat{X}_{W_a})$ and the interval $[e,w]_R$ for the right weak order of $W_a$.
\end{proposition}

\begin{proof}
Once again Theorem \ref{TH central} tells us that this map is a bijection. Using Remark \ref{remark} we have 
$
X_{W_a}[\lambda] \leq_S X_{W_a}[\lambda'] \Longleftrightarrow \lambda \leq_S \lambda' \Longleftrightarrow w_{\lambda} \leq_R w_{\lambda'}
$.
This shows that this map is a morphism of posets. Moreover, this also shows that $\lambda, \lambda'$ define an interval for $\leq_S$ if and only if $w_{\lambda}$ and $w_{\lambda'}$ define an interval for the right weak order. Therefore, we only have to show that $[0, \lambda_w]_S$ is a well defined interval.
It is obvious that 0 is less than any other admitted vector and that it is the only one satisfying this condition. Thanks to Proposition \ref{max}, $H^0(\widehat{X}_{W_a})$ has a unique maximal element which is $X_{W_a}[\lambda_w]$ (identified with $\lambda_w$). It follows then that $[0, \lambda_w]_S$ does define an interval.
\end{proof}

We are now ready to prove the main theorem.

\begin{proof} [Proof of Theorem \ref{lattice SD}]

It suffices to show that $[e,w]_R$ is a semidistributive lattice. This stems from Theorem 8.1 of \cite{NRDS} that states that every interval in the right weak order of any infinite Coxeter group is a semidistributive lattice.
\end{proof}

\noindent \textbf{Acknowledgements}.
I thank Christophe Hohlweg, Christophe Reutenauer and Hugh Thomas for answering many questions and providing many helpful comments that help us to  improve this paper. I'm also grateful to Antoine Abram and Nicolas England for valuable discussions. Finally, I thank the referee for his/her thorough reading.

This work was partially supported by NSERC grants and by the LACIM.

\bibliographystyle{plain}
\bibliography{components_poset.bib}

\end{document}